\documentclass{amsart}
\newtheorem{theorem}{Theorem}[section]
\newtheorem{lemma}[theorem]{Lemma}
\theoremstyle{definition}
\newtheorem{definition}[theorem]{Definition}

\theoremstyle{remark}
\newtheorem{remark}[theorem]{Remark}

\numberwithin{equation}{section}

\newcommand{\xp}{\frac{2}{1+\beta}}
\newcommand{\eo}{\epsilon_0}

\newcommand{\al} {\alpha}
\newcommand{\ba} {\beta}
\newcommand{\de} {\delta}
\newcommand{\ga} {\gamma}

\newcommand{\Om} {\Omega}
\newcommand{\pa} {\partial}
\newcommand{\De} {\Delta}
\newcommand{\la} {\lambda}
\newcommand{\si} {\sigma}

\newcommand{\rar}{\rightarrow}

\newcommand{\noi} {\noindent}

\newcommand{\na} {\nabla}

\newcommand{\Aa}{\mathcal{A}}

\newcommand{\ci}{C_0^1(\overline{\Om})}

\usepackage{amsmath,stackengine}
\usepackage{xcolor}
\title[Infinite semipositone problem for p-q Laplacian]
{On a class of infinite semipositone problems for (p,q) Laplace operator}	

\author{R. Dhanya, R. Harish, Sarbani Pramanik}
\address{School of Mathematics, Indian Institute of Science Education and Research Thiruvananthapuram, Maruthamala, Thiruvananthapuram, 695551}



\begin{document} 
\begin{abstract}
We analyze a non-linear elliptic boundary value problem, that involves  $(p, q)$ Laplace operator, for the existence of its positive solution in an arbitrary smooth bounded domain. The non-linearity here is driven by  a continuous function in $(0,\infty)$ which is singular, monotonically increasing and eventually positive. We prove the existence of a positive solution of this problem  using a fixed point theorem due to Amann\cite{amann1976fixed}. In addition, for a specific nonlinearity we derive that the obtained solution is maximal in nature. The main results obtained here are first of its kind for a $(p, q)$ Laplace operator in an arbitrary bounded domain.

\end{abstract}	
\maketitle	
\section{Introduction}
We consider the following boundary value problem
\begin{align}\label{problem}
    -\De_p u - \De_q u &= \la \frac{f(u)}{u^\ba} \ \text{in} \ \Om \nonumber \\
    u &>0 \ \text{in} \ \Om \tag{$P_\lambda$} \\
    u &=0 \ \text{on} \ \pa\Om\nonumber
\end{align}
where  $\Om$ is a bounded domain in $\mathbb{R}^N$  with a smooth boundary,  $0<\ba<1$, $1<q<p<\infty$ and $\la>0$ is a parameter. The term $\De_s u:= div(|\na u|^{s-2}\na u)$. We make the following assumptions on $f$:
\begin{enumerate}
    \item[(H1)] $f:[0,\infty) \rightarrow \mathbb{R}$ is a continuous monotonically increasing function with $f(0)<0$ and $\frac{f(t)}{t^{\ba}}$ is monotonically increasing.
    \item[(H2)]  $\frac{f(t)}{t^\beta}>0$ for some $t>0.$
\end{enumerate}
Note that here $\lim_{s \rightarrow 0+} \frac{f(s)}{s^\ba} = - \infty$ and such problems are called infinite semipositone problems. Establishing the existence of a positive solution to such problems are challenging due to the presence of the singular term which is negative. 
Moreover, due to the non-homogeneous nature of the associated operator, many of the well known methods are not directly applicable.\par

The motivation to study the non-homogeneous operator $(p, q)$ Laplacian emanates from pure mathematical interest and also due to its  applications in several branches of science. The operator $(p, q)$ Laplacian, a generalisation of the $p$-Laplacian, is defined as an approximation of the mean curvature operator in Minkowski space which arises in the study of Born-Infeld equation, see \cite{bonheure2016electrostatic} and \cite{pomponio2018some}. The structure of the energy functional associated with $(p, q)$ Laplacian drives its application into diverse fields, for instance in the study of the geometry of anisotropic materials, modelling of elementary particles (\cite{benci1998soliton}, \cite{derrick1964comments}), chemical reaction design (\cite{maki1980r}, \cite{cherfils2005stationary}), elasticity theory (\cite{zhikov1987averaging}), biophysics (\cite{fife1979lecture}), plasma physics (\cite{wilhelmsson1987explosive}) etc. are a few to name. Though the non-homogeneity of the operator makes the mathematical treatments strenuous,  it is  well studied in literature, see \cite{cherfils2005stationary},  \cite{marano2017some},  \cite{faria2014comparison}, \cite{mugnai2014wang},\cite{figueiredo2013existence}, \cite{gasinski2020existence} and the references therein.

In general,  the problem 
\begin{align}\label{problemsemi}
    -\De_p u - \De_q u &=  h(u) \ \text{in} \ \Om \nonumber \\
    u &>0 \ \text{in} \ \Om \\
    u &=0 \ \text{on} \ \pa\Om\nonumber
\end{align}
is called semipositone problem if $h$ is a monotonically increasing  sign-changing continuous function such that $h(0)\in (-\infty,0).$  If $h(0)=-\infty,$ we call this an infinite semipositone problem. Semipositone problems exhibit certain interesting phenomena which are not expected for positone problems (i.e. when $h(0)>0$). For example,  non-negative solutions for semipositone problems can have interior zeros (see \cite{castro1988non})  due to which  proving the existence of positive solutions to semipositone problems are comparatively harder. However topological tools such as degree theory, fixed point methods, bifurcation theory  etc. are mainly found reliable in obtaining the positive solution to semipositone problems (see \cite{chhetri2015existence}, \cite{hai2010singular}, \cite{hai2012positive} and \cite{dhanya2018positive}). But all these methods depend strongly on the qualitative properties of the solutions, such as regularity, asymptotic behavior, a priori estimates,  stability etc. Due to the non-homogeneity of the operator, such results were not available for the $(p,q)$ Laplace operator until recently. In \cite{giacomoni2021sobolev}, Giacomoni et.al. proved $C^{1,\alpha}$ boundary regularity for the solutions of singular elliptic problems involving $(p,q)$ Laplacian and  in \cite{Indulekha} authors proved an asymptotic behavior of solutions for such problems. In light of these latest developments, we address the question of the existence of a positive solution to an infinite semipositone problem using sub-super solution approach.

In the recent years, semipositone problems for non-homogeneous operators are gaining considerable attention. Das et al. in \cite{das2020existence} proved the existence of a positive solution for such a problem in an arbitrary domain when the reaction term is non-singular. In \cite{sim2021positive}, authors study an infinite semipositone problem for $(p,q)$ Laplacian in an interval using a fixed point theorem. Hai and co-authors in \cite{hai2023positive} have recently studied a similar equation in a ball by assuming the radial nature of the solution and converting it to an ODE. They have used the representation formula for the solution to obtain the required estimates and prove the existence of a positive solution. The approach discussed there cannot be extended to an  arbitrary bounded open set as the method entirely depends on the symmetry of the domain. Nevertheless, Hai et al., in a series of papers  \cite{alotaibi2020singular}, \cite{hai2013singular}, \cite{hai2010singular}, \cite{hai2011class} and \cite{hai2012infinite}, has considered similar problems for $p$-Laplace operator in arbitrary bounded open sets in $\mathbb{R}^N$. Interested readers may also refer to \cite{alves2020existence}, \cite{anello2017two},\cite{choubin2013positive}, \cite{goddard2013existence},  and \cite{khodja2016positive} for semilinear or quasilinear elliptic problems with indefinite sign non-linearity in bounded domains. The novelty of our paper lies in the fact that the results we obtain are first of its kind for a $(p,q)$ Laplace operator in an arbitrary domain with a singular non-linearity.

In Section 2, we focus on a general theory regarding the existence of positive minimal and maximal solutions and in Section 3 and 4, we discuss specific examples.
\begin{theorem}\label{thm1}
Suppose $f$ satisfies (H1) and (H2). Also assume that there exists an ordered pair $[\psi,\phi]$ of sub- and supersolutions for \eqref{problem} such that $\psi, \phi \in C^1_0(\overline{\Om})$, there exists a positive constant $c_0$ satisfying $c_0 d(x) \leq \psi(x) \leq \phi(x)$ and $\|\psi\|_\infty>\ba_0,$ where $\beta_0=\inf\{t>0: \frac{f(t)}{t^\beta} >0\}$. Then the problem \eqref{problem} admits a maximal and a minimal solution within the ordered interval $[\psi,\phi]$.
\end{theorem}
The above theorem is rather a straight forward extension of results in \cite{dhanya2015three} to the $(p, q)$ Laplace operator, thanks to the global regularity result proved in \cite{giacomoni2021sobolev}. Rather the nontrivial step is to prove the existence of sub-supersolutions which satisfies the assumptions mentioned in the above theorem. Non-linearity combined with the non-homogeneity of the operator makes this task harder. In this context we wish to emphasize that the results we prove here are not yet known even for the p-Laplacian operator.\par
In the next theorem under two more hypotheses on the non-linear function $f$ we prove the existence of an ordered pair of sub-supersolutions bounded below by $c_0 d(x).$ We assume that
\begin{enumerate}
    \item[(H3)] There exist constants $\si\in (0,\beta+1)$ and $A>0$ such that $f(s)\geq A s^{\sigma}$ for $s>>1$.
    \item[(H4)] There exist $\ga>0$ and $B>0$ such that $\ba \leq \ga< \ba+1$ and $f(s) \leq Bs^\ga$ for $s\geq 0$.
\end{enumerate}
The condition (H4) implies $\lim_{s\to\infty} \frac{f(s)}{s^{\ba+1}}=0.$
\begin{theorem}\label{thm2}
Suppose that $f$ satisfies hypothesis (H1)-(H4) and $2<q<p<\infty$. Then for $\lambda>>1$, there exists a positive solution for the problem \eqref{problem}.
\end{theorem}
Motivated by the constructions given in \cite{dhanya2015three} and \cite{das2020existence}, Theorem \ref{thm2} is proved in Section 3 by the sub-supersolution approach. The additional hypothesis $q>2$ helps us to circumvent the non-homogeneity of the operator in a clever manner.  If $p=q,$  this restriction of $q>2$ is not necessary as the operator is no more inhomogeneous.  

In section 4 we prove that \eqref{problem} admits a maximal solution for a specific non-linear function $f$. 
The existence of a global minimal or maximal solutions is the foremost step in the bifurcation analysis of a semilinear elliptic problem, see \cite{dhanya2018positive} where a similar problem is studied in case of Laplacian. In \cite{dhanya2018positive}, author has shown the existence of continuum of a branch of positive solutions emanating from infinity and a multiplicity result under a specific hypothesis. This phenomenon of multiple positive solutions for non-linearities with indefinite sign is difficult to obtain and is understood completely only in the case of Laplacian in dimension one \cite{diaz2009branches} and for the strictly semipositone problems in higher dimension\cite{CASTRO1994425}. We believe that Theorem \ref{max_thm} along with the understanding of the weighted eigenvalue problem for a p-Laplacian can throw some light into the multiplicity results for the associated quasilinear semipositone problem. 

Our second main result in this paper is proved in Theorem \ref{max_thm}, Section 4.  Before stating the theorem we shall recall the function space where we seek a positive solution and its properties as given in \cite{amann1976fixed}. Let $\mathcal{P}=\{u\in C^1_0(\overline{\Omega}) : u\geq 0 \mbox{ in } \Omega\}$ denote the positive cone in $C^1_0(\overline{\Omega}).$ Then the interior of $\mathcal{P}$ denoted by $\mathcal{P}^\circ$ consists of  set of all $\{u\in C^1_0(\overline{\Omega}): \frac{\pa u}{\pa \nu}<0\}$ where $\nu$ denotes the outward normal vector on $\pa\Omega.$ We shall prove that \eqref{problem} admits a maximal positive solution in $\mathcal{P}^\circ.$
\begin{theorem}\label{max_thm}
  Let $\beta\in (0,1) , 2<q<p<\infty,$ $\sigma\in (0,q-1)$ and $f(s)=(s^{\si+\beta}-1).$ Then the  problem \eqref{problem} admits maximal solution $u_\la\in \mathcal{P}^\circ $ for large $\lambda$. Also, \eqref{problem} does not admit any positive solution when $\la$ is small.   
\end{theorem}
 Towards the end of this article, we make an attempt to understand the exact behavior of the solution of a "sublinear" elliptic problem given in \eqref{sup_est}. The estimate we obtain in Theorem \ref{est} is of independent interest and the result we obtain may be applicable to several other non-linear elliptic problems involving $(p,q)$ Laplacian. Results we obtain is restricted to the case $2<q<p<\infty.$
\section{Preliminaries on General Theory}
\begin{definition}
    We define $w\in W^{1,p}_0(\Om) , w>0 $ to be a weak solution of \eqref{problem} if
    \begin{equation*}
        \int_\Om |\na w|^{p-2}\na w \cdot \na \phi\ dx + \int_\Om |\na w|^{q-2}\na w \cdot \na \phi\ dx = \la \int_\Om \frac{f(w)}{w^\ba} \phi\ dx
    \end{equation*}
    holds for all $\phi \in C_c^\infty(\Om)$.
\end{definition}
\textbf{Notations:} \begin{enumerate}
    \item By $d(x)$ we denote the distance of $\pa \Om$ from the point $x$, i.e, \[d(x)=inf\{\|x-y\|:y\in \pa \Om\}\]
    \item We write $\Om_\epsilon$ for the set $\{x \in \Om: d(x)<\epsilon\}$.
\end{enumerate}
    \begin{lemma}\label{exstmin}
        Given $h\in C(\Omega)$ such that $|h(x)|\leq  \frac{C}{d(x)^{\ba}},$ there exists a unique $w\in  W^{1,p}_0(\Omega)$ solving
         \begin{align}\label{w}
             -\De_p w- \De_q w &= h \ \text{in} \ \Om \nonumber \\
             w &=0 \ \text{on} \ \pa \Om
         \end{align}
        in the weak sense.
    \end{lemma}
    
    \begin{proof}
        We will prove the existence of the solution by the direct minimization technique. The energy functional $E: W^{1,p}_0(\Omega)\rightarrow \mathbb{R}$ corresponding to  $ -\De_p w -\Delta_q w=h(x) $ is : \\
        \begin{equation}
            E(u)= \frac{1}{p} \int_{\Om}^{} |\nabla u|^{p} dx + \frac{1}{q} \int_{\Om}^{} |\nabla u|^{q} dx - \int_{\Om}^{} hu\ dx
        \end{equation}
        Using H\"older's inequality and Hardy's inequality we get 
        $$E(u_n) \geq \frac{1}{p}\|\na u_n\|^p_{L^p(\Om)} + \frac{1}{q}
        \|\na u_n\|^q_{L^q(\Om)}- C'\|\na u_n\|_{L^p(\Om)} $$
       and hence $E$ is coercive. In a standard way, we can also show that $E$ is weakly lower semi continuous  on the space $W^{1,p}_0(\Omega).$ Thus $E$ possesses a global minimizer $w \in W_0^{1,p}(\Om)$ and by weak comparison principle this minimizer $w$ is the unique weak solution of the required problem.
    \end{proof}

    \begin{lemma}\label{uniform_bound}
        Let $h$ and $w$ be as in Lemma \ref{exstmin} and $\beta\in (0,1).$ Then there exist constants $\al \in (0,1)$ and $M>0$, depending only on $C, \ba$ and $ \Om$ such that $w \in C^{1, \al}(\overline{\Om})$  and $\|w\|_{C^{1, \al}(\overline{\Omega})}<M. $
    \end{lemma}
    
    \begin{proof}
        Let $u$ be the solution of
        \begin{align}\label{u}
            -\De u &=h \ \text{in} \ \Om \nonumber\\
            u &=0 \ \text{on} \ \pa\Om
        \end{align}
         and by lemma 3.1 from \cite{hai2011class}, there exist constants $\al \in (0,1)$ such that $u \in C^{1, \al}(\overline{\Om}).$\\ 
         Also, let $\tilde{w}$ be the solution of
         \begin{eqnarray}
              -\De_p \tilde{w}- \De_q \tilde{w} = |h| \ \text{in} \ \Om \nonumber \\
            \tilde{w}=0 \ \text{on} \ \pa \Om \nonumber
         \end{eqnarray}
         Then, by weak comparison principle and proposition 2.7 from \cite{giacomoni2021sobolev}, there exists a constant $C_1$ such that $w \leq \tilde{w} \leq C_1 d(x) $ in  $\Om$. Again
         \begin{equation*}
             -\De_p(-w) -\De_q(-w)= -h \leq |h|= -\De_p \tilde{w}- \De_q \tilde{w} \ \text{in} \ \Om
         \end{equation*}
         and $-w=\tilde{w}=0 \ \text{on} \ \pa\Om .$  By a similar reasoning $-w \leq \tilde{w} \leq C_1 d(x) $ in $ \Om$ and therefore, we have,
         \begin{equation}
             |w| \leq C_1 d(x) \mbox{ in }  \ \Om
         \end{equation}
         Then from \eqref{w} and \eqref{u}, we have,
         \begin{equation}
             div \bigl( |\na w|^{p-2} \na w + |\na w|^{q-2} \na w - \na u \bigr)=0 \mbox{ in } \ \Om
         \end{equation}
         and $w=0 \ on \ \pa\Om.$   Now we conclude the proof by applying Lieberman's boundary regularity result given in  Theorem 1.7 of \cite{lieberman1991natural}.
   \end{proof} 

    In this section, from now on, we assume that \eqref{problem} admits a positive subsolution $\psi \in C^{1}_0(\overline{\rm\Om}).$ Also we assume that  $\psi(x) \ge c_0 d(x)$ for some positive constant $c_0$  and $\|\psi\|_\infty > \beta_0 $ where $\beta_0$ is given in Theorem \ref{thm1}. We define the set 
    \begin{center}
       $\Aa = \{u \in C^{1}_0(\overline{\Om}) : u(x) \ge \psi(x)\}$
    \end{center} 
Fix a $\la>0$ and let $u \in \Aa $ and let $g(u) := \la \frac{f(u)}{u^\beta}$. Then  $|g(u)| \leq \frac{c_1}{d(x)^\beta}$ for some $c_1>0$.  Let $w$ be the unique weak solution of $-\De_p w -\De_q w= g(u)$ in $\Om$,  $w=0$  on  $\pa \Om$ which is now guaranteed to exist by Lemma \ref{exstmin}. Since g is monotonically increasing and $u \ge \psi$, we have 
     \begin{center}
         $-\De_p w -\De_q w= g(u) \ge g(\psi) \ge -\De_p \psi -\De_q \psi$ in $\Om$ and   $ w = \psi = 0 $ on $\pa \Om$
    \end{center}
    Now by weak comparison principle, we have $w \ge \psi$. Also, by Lemma \ref{uniform_bound}, $w \in C^{1,\al}(\overline{\Om})$, form some $\al \in (0,1)$. Hence, we define the map $T_g$ as given below.

    \begin{definition}
        We define the map $T_g : \Aa \rightarrow \Aa $ as $T_g(u)=w$ iff $w$ is the weak solution of
        \begin{align}
            -\De_p w -\De_q w &= g(u) \ \text{in} \ \Om \nonumber\\
            w &=0 \ \text{on} \ \pa\Om
        \end{align}
        Clearly, $T_g$ is a well-defined map from $\Aa$ to itself. Next, we show that the map $T_g: \Aa \rightarrow \ci $ is completely continuous.
    \end{definition}

    \begin{lemma}\label{comp_cont}
        $T_g : \Aa \rightarrow \ci $ is completely continuous.
    \end{lemma}
    
    \begin{proof}
        Fix $u \in \Aa ,$ then $u(x) \ge c_0 d(x)$ and $\|u\|_{\infty}>\beta_0.$ Let $h\in \ci$ such that $\|h\|_{C_0^{1}(\Bar{\Om})}<\delta.$  Choose $\delta$ small enough so that $u(x)+h(x) \ge \frac{c_0}{2} d(x)$ and $\|u+h\|_\infty > \beta_0$. Let $w_h$ be the unique solution of 
 \begin{align}
      -\De_p w_h -\De_q w_h &= g(h+u)\ \text{in}\ \Om \nonumber \\ 
     w_h &= 0\ \text{on}\ \pa \Om
 \end{align}
 and $w$ be the unique solution of 
 \begin{align}\label{eqngu}
     -\De_p w -\De_q w &= g(u)\ \text{in}\ \Om \nonumber \\
     w &= 0\ \text{on}\ \pa \Om  
 \end{align}
       Note that $g(u + h) \rightarrow g(u)$ pointwise as $\|h\|_{C_0^{1}(\Bar{\Om})} \rightarrow  0$. As
 $|g(u+h)| \le\ C_1 d(x)^{- \beta}$, by Lemma \ref{uniform_bound} there exists an $\alpha \in (0,1)$ such that $\|w_h\|_{C^{1,\alpha}(\Bar{\Om})} < M $. Now by Ascoli Arzela, theorem upto a subsequence, $w_{h_k}\rightarrow \tilde{w}$ in $C^{1,\alpha-\epsilon}(\overline{\Omega}).$ Also for all $\phi \in W_0^{1,p}(\Om),$
\begin{equation}\label{hkeqn}
      \int_{\Om}^{} |\nabla{w_{h_{k}}}|^{p-2} \nabla{w_{h_{k}}} \cdot \nabla{\phi}\ dx + \int_{\Om}^{} |\nabla{w_{h_{k}}}|^{q-2} \nabla{w_{h_{k}}} \cdot \nabla{\phi}\ dx = \int_{\Om}^{} g(u+h_k)\ \phi\ \text{dx} 
  \end{equation}
  Passing through the limit in the equation \eqref{hkeqn}, we get 
   \begin{equation}
      \int_{\Om}^{} |\nabla{\tilde{w}}|^{p-2} \nabla{\tilde{w}} \cdot \nabla{\phi}\ dx + \int_{\Om}^{} |\nabla{\tilde{w}}|^{q-2} \nabla{\tilde{w}} \cdot \nabla{\phi}\ dx = \int_{\Om}^{} g(u)\phi\ \text{dx}
  \end{equation}
  Thus $\tilde{w}$ is a weak solution of \eqref{eqngu} and by uniqueness $w=\tilde{w}$. In a standard way, we can show that every subsequence of the original sequence $\{w_h\}$ converges to $w$ and hence $T_g: \Aa \rightarrow C^{1,\alpha-\epsilon}_0(\overline{\Omega})$ is continuous.  
  In other words, $T_g(u+h) \rightarrow T_g(u)$ in  $C^{1,\alpha-\epsilon}(\overline{\Omega})$ as $\|h\|_{C_0^{1}(\Bar{\Om})} \rightarrow  0$. Using the fact that $C^{1,\alpha-\epsilon}_0(\Bar{\Om}) \subset \subset C^{1}_0(\Bar{\Om}),$ we conclude that $T_g: \Aa \rightarrow {C_0^{1}(\Bar{\Om})}$ is completely continuous. 
    \end{proof}
\textbf{Proof of Theorem \ref{thm1}:} Since $g$ is monotonically increasing, we know that $T_g$ maps $X$ into itself where $X=\{u\in C^1_0(\overline{\Omega}): \psi\leq u \leq \phi\}.$ Also, the map $T_g: X \rightarrow X $ is completely continuous. Now our result  follows due to the celebrated Theorem 6.1 of Amann\cite{amann1976fixed}.\hfill\qed.
    \section{Sub-Super solution construction}
    
    In this section, we consider $2<q<p<\infty$ and focus on constructing an ordered pair sub and supersolutions for the problem \eqref{problem} belonging to the set $\Aa.$ Here we assume $f$  satisfies all the hypotheses $(H1)-(H4)$ mentioned in the introduction.
    \begin{lemma}\label{lemma_sub}
     Assume $f$ satisfies (H1)-(H3). Let $\phi_1$ be the first eigen function of $-\De$ with zero Dirichlet boundary condition. Define \[\psi := \la^r (\phi_1 +\phi_1^\xp)\] where $\frac{1}{p-1+\ba}< r< \frac{1}{p-1+\ba-\si}$. Then, for large $\la$, $\psi \in \ci$ is a subsolution for the problem \eqref{problem} satisfying $\psi(x) \geq c_\lambda d(x).$ 
    \end{lemma}
    
    \begin{proof}
    From the definition, since $\beta<1$ it is clear that $\psi \in \ci$ and $\psi(x) \geq c_\la d(x)$. We only need to show that $\psi$ is a subsolution for (\ref{problem}) when $\la$ is large. We have, for $s=p,q$, 
    \begin{equation*}
        \begin{split}
            \De_s \bigl(\phi_1 + \phi_1^\xp \bigr) & = div \Bigl[\Bigl(1+\xp\phi_1^{\xp-1}\Bigr)^{s-1}|\na\phi_1|^{s-2} \na\phi_1 \Bigr]\\
            & =\Bigl(1+\xp\phi_1^{\xp-1}\Bigr)^{s-1} div \Bigl(|\na\phi_1|^{s-2} \na\phi_1 \Bigr) \\
            & + \xp\Bigl(\xp-1 \Bigr)\Bigl(s-1\Bigr)|\na\phi_1|^s \Bigl(1+\xp\phi_1^{\xp-1}\Bigr)^{s-2} \phi_1^{\xp-2}\\
            & = \Bigl(1+\xp\phi_1^{\xp-1}\Bigr)^{s-1} [\De_s \phi_1] \\
            & + \xp\Bigl(\xp-1\Bigr)\Bigl(s-1\Bigr)|\na\phi_1|^s \Bigl(1+\xp\phi_1^{\xp-1}\Bigr)^{s-2} \phi_1^{\xp-2}\\
        \end{split}
    \end{equation*}
Thus we write, 
    \begin{equation}\label{deltas}
         -\De_s \psi=\la^{r(s-1)}\phi_1^{\frac{-2\beta}{1+\beta}} L_s(\phi_1)
    \end{equation}
    where 
\begin{multline} \nonumber 
L_s(\phi_1)= \Bigl(1+\xp\phi_1^{\frac{1-\beta}{1+\beta}}\Bigr)^{s-1}\phi_1^{\frac{2\beta}{1+\beta}}[-\De_s \phi_1]\\
        - \xp\Bigl(\frac{1-\beta}{1+\beta}\Bigr)\Bigl(s-1\Bigr)|\na\phi_1|^s \Bigl(1+\xp\phi_1^{\frac{1-\beta}{1+\beta}}\Bigr)^{s-2}
        \end{multline}
    Now as $s>2,$ we have $\De_s \phi_1 \in L^\infty(\Om)$  and by Vasquez (\cite{vazquez1984strong}), $|\na\phi_1|>0$ on $\pa\Om.$ Thus there exists $\de>0$ and $m>0$ such that for in $x\in \Om_{\de}$,
    \begin{equation}\label{lphineg}
        L_p(\phi_1)<-m \;\;\; \text{ and }\;\;\;\; L_q(\phi_1)<0
    \end{equation}
    Using (\ref{deltas}) and $(\ref{lphineg})$ we have 
    \begin{equation}
        \displaystyle -\Delta_p \psi-\Delta_q \psi \leq -\frac{ m\, \la^{r(p-1)}}{\phi_1^{\frac{2\beta}{1+\beta}}}\leq -\frac{m \, \la^{r(p-1+\ba)}}{\psi^\ba} \;\; \mbox{ for all } x\in\Om_\delta.
    \end{equation}
    The last inequality is obtained directly from the definition of $\psi.$ Now that $f(0)$ is negative and f is monotone increasing, so for $r>\frac{1}{p-1+\ba}$ and $\la>>1,$ we have, in $\Om_{\de}$,
    \begin{equation*}
      -\frac{m \, \la^{r(p-1+\ba)}}{\psi^\ba}  \leq \  \frac{\la f(0)}{\psi^{\ba}} \leq \ \frac{\la f(\psi)}{\psi^{\ba}}
    \end{equation*}
   Hence,
    \begin{equation}\label{sub1}
        -\De_p\psi- \De_q\psi \leq \frac{\la f(\psi)}{\psi^{\ba}} \ \text{in} \ \Om_\de
    \end{equation}
    In $\Om \setminus \Om_\de$, there exists a positive constant $\mu$ such that $\phi_1>\mu$ and thus $\psi>\la^r {\mu}$. Also using the monotonicity of $f$ and the assumption (H3) on $f$, we have for $\la>>1$,
    \begin{equation*}
        f(\psi)> \ f(\la^r {\mu}) \geq \ A\la^{r\si}{\mu}^\si
    \end{equation*}
    Further, without loss of generality, we assume $\|\phi_1\|_\infty=1$. Therefore, in $\Om \setminus \Om_\de$, we have, for $s=p,q$,
    \begin{equation}
        -\De_s\psi \leq \frac{\la^{r(s-1)}c_s}{\overline{\mu}} \leq \frac{\la[A\la^{r\si}{\mu}^\si]}{2\psi^\ba} \leq \frac{\la f(\psi)}{2\psi^\ba}
    \end{equation}
    whenever $r<\frac{1}{s-1+\ba-\si}$ and $c_s=(1+\xp)^{s-1} \|\De_s \phi_1\|_\infty$ and $\overline{\mu}=\mu^{\frac{2\ba}{1+\ba}}$.
    Since $q<p$, by taking $r<\frac{1}{p-1+\ba-\si}$ we have,
    \begin{equation}\label{sub2}
        -\De_p\psi- \De_q\psi \leq \frac{\la f(\psi)}{\psi^{\ba}}, \ \text{in} \ \Om \setminus \Om_\de
    \end{equation}
    Now from (\ref{sub1}) and (\ref{sub2}) we have the required result.

    \end{proof}
    \begin{lemma}\label{superlemma}
        Assume $f$ satisfies (H1), (H2) and (H4).Then the problem \eqref{problem} admits a super-solution for all $\la>0.$
    \end{lemma}
    
    \begin{proof}
    Let $R>0$ be such that $\overline{\Om} \subset B_R(0)$ where $B_R(0)$ is the open ball of radius $R$ around origin and $e$ be the unique solution of 
        \begin{align}\label{e}
            -\De_p e &=1 \ in \ B_R(0) \nonumber \\
            e &=0 \ on \ \pa B_R(0)
        \end{align}
        We know that $e(x)$ must be a radial function and  $e(x)= R^{\frac{p}{p-1}}\,\eta(r)$ where 
        \begin{equation}
           \eta(r)= \frac{{1-\bigl(\frac{r}{R}\bigr)^{p'}}}{p'}.
        \end{equation}  
        Using the hypothesis $(H4)$ we can choose a constant $m(\la)>>1$ such that 
        \[m(\la)^{p-1+\ba-\ga} \geq \la Be^{\ga-\ba} \]
        We claim that the function $\phi:= m(\la)e$ is a supersolution for \eqref{problem}.\\
    Again using (H4),
        \begin{equation}\label{psuper}
            -\De_p \phi=m(\la)^{p-1} \geq \frac{B\la \bigl(m(\la)e\bigr)^\ga}{\bigl(m(\la)e\bigr)^\ba} \geq \frac{\la f(\phi)}{\phi^\ba}, \ \text{in} \ \Om
        \end{equation}
Also,   \begin{equation}\label{ineq}
            -\De_q \eta=-\bigl(|\eta'(r)|^{q-2} \eta'(r)\bigr)' = \Bigl(\frac{r^{(p'-1)(q-1)}}{R^{p'(q-1)}}\Bigr)' \geq 0 \ \text{in} \ B_R(0),
        \end{equation}
    and $-\De_q \phi = -k \De_q \eta$ in $\Om$ for some $k>0.$ Thus, from \eqref{psuper} and \eqref{ineq} we have
        \begin{equation*}
            -\De_p \phi - \De_q \phi \geq -\De_p \phi \geq \frac{\la f(\phi)}{\phi^\ba} \ \text{in} \ \Om\;\;\;\; \phi\geq 0 \text{  on } \pa\Omega. 
        \end{equation*}
       Hence, $\phi$ is a super-solution for \eqref{problem}.
    \end{proof}
\noi \textbf{Proof of Theorem \ref{thm2}:}    We note that under the hypotheses $(H1) - (H4)$ using Lemma \ref{lemma_sub} and Lemma \ref{superlemma}, the problem (\ref{problem}) admits a pair of sub-supersolutions for all $\lambda\geq \la_0.$ 
 Also $\phi > 0$ in $\overline{\Om}$ and hence $m(\la)$ can be chosen large enough so that $\phi \geq \psi$ in $\overline{\Om}$. Thus the ordered pair $[\psi, \phi]$ satisfies all the conditions of Theorem \ref{thm1} so that \eqref{problem} admits a minimal and maximal solution within the ordered interval $[\psi, \phi].$ \hfill\qed.
\section{Existence of Maximal Solution}
In this section, we are interested in the following boundary value problem:
\begin{align}\label{max_prblm}
    -\De_p u - \De_q u &= \la (u^\si-\frac{1}{u^\ba}) \ \text{in} \ \Om \nonumber \\
    u &>0 \ \text{in} \ \Om\tag{$\tilde{P}_\la$} \\
    u &=0 \ \text{on} \ \pa\Om \nonumber
\end{align}
where $0< \beta <1 ,\,  0<\sigma<q-1$ and $2<q<p<\infty.$ It is evident that $f(s)= s^{\sigma+\ba}-1$ satisfies the hypotheses $(H1)- (H2)$. We define the solution set $\mathcal{S}$ to be 
$$\mathcal{S}=\{u\in C^1_0(\overline{\Omega}) : u(x)\geq k d(x) \text{ for some } k>0\}.$$
Our aim in this section is to prove the existence of a positive maximal (global maximal) solution for the Dirichlet problem (\ref{max_prblm}) belonging to the set $\mathcal{S}.$
We can verify that  $\Psi_\la =\la^r (\phi_1 +\phi_1^\xp)$  defined in Lemma \ref{lemma_sub} is a positive subsolution of \eqref{max_prblm}. If there exists a supersolution $\Phi_\la$ such that $\Psi_\la\leq \Phi_\la,$ then 
 Theorem \ref{thm2} can be applied and it guarantees a maximal solution in the ordered interval $[\Psi_\la,\Phi_\la].$ Our aim in this section is to redefine our supersolution appropriately so that the resulting solution obtained by Theorem \ref{thm2} is in fact a maximal (global) solution. We make the following definition.
\begin{definition} We say that $\Phi_\la$ is a global supersolution to \eqref{max_prblm} if
\begin{itemize} \item[(i)] $\Phi_\la$ is a supersolution of \eqref{max_prblm} and 
\item[(ii)] if $u$ is any solution of \eqref{max_prblm}, then $u\leq \Phi_\la.$
\end{itemize} 
\end{definition}
We make the following observation which follows from Amman's fixed point theorem. 
\begin{remark} Suppose there exists a global supersolution to \eqref{max_prblm} which we call as $\Phi_\lambda.$ Also assume that $\Psi_\lambda\leq \Phi_\la$ where $\Psi_\la =\la^r (\phi_1 +\phi_1^\xp)$ is a subsolution of 
\eqref{max_prblm}. Then there exists a global maximal solution for \eqref{max_prblm} within the ordered interval $[\Psi_\la,\Phi_\lambda].$
\end{remark}
\noi We now prove that the solution of the following BVP 
\begin{align}\label{sup_est}
    -\De_p z - \De_q z &= \la z^\si \ \mbox{in} \ \Om\nonumber \\
    z &>0 \mbox{ in } \Omega\\
    z &=0 \ \mbox{on} \ \pa\Om\nonumber
\end{align}
 is a global supersolution to \eqref{max_prblm}. The existence and uniqueness of the positive solution of (\ref{sup_est}) can be proved via a global minimization technique and Lemma \ref{comparison}. Let us denote ${\Phi}_\la$ to be the unique solution of \eqref{sup_est} and since
$-\De_p \Phi_\la - \De_q \Phi_\la \geq \la (\Phi_\la^\si-\frac{1}{\Phi_\la^\ba})$,  $\Phi_\la$ serves as a supersolution for \eqref{max_prblm}. Using the comparison principle proved in Lemma \ref{comparison}, we can show that if $u$ is any solution of \eqref{max_prblm}, then $u\leq \Phi_\lambda.$ Thus, $\Phi_\la$ is a global supersolution of \eqref{max_prblm}.



 \noi\textbf{ Proof of Theorem \ref{max_thm}:} We first show that for large $\la$, $\Psi_\la \leq \Phi_\la$ in $\Om$. Since $\Psi_\la$ is a positive subsolution of \eqref{max_prblm}, we have,
 \begin{equation}
     -\De_p \Psi_\la -\De_q \Psi_\la \leq \la (\Psi_\la^\si- \frac{1}{\Psi_\la^\ba}) \leq \la \Psi_\la^\si\ \mbox{in}\ \Om
 \end{equation}
 and $\Psi_\la=0$ on $\pa\Om$, i.e, $\Psi_\la$ is a positive subsolution of \eqref{sup_est} also. Consequently, by lemma \ref{comparison}, we have $\Psi_\la \leq \Phi_\la$ in $\Om$ for $\la>>1$.

 
Now Theorem \ref{thm1} ensures the existence of a positive maximal solution for the problem \eqref{max_prblm} within the ordered interval $[\Psi_\la, \Phi_\la]$. Let us call this maximal solution $u_\la$. Again if $u$ is any other positive solution of \eqref{max_prblm}, then Lemma \ref{comparison} gives $u \leq \Phi_\la$ in $\overline{\Om}$. By the construction of $u_\la $, clearly $u\leq u_\lambda$ i.e , $u_\la$ is the global maximal solution of \eqref{max_prblm}.\\
Next suppose $\lambda$ is small enough and there exists a positive solution $u_\lambda$ for \eqref{max_prblm}. Since $\la\mapsto \Phi_\la$ is monotone increasing, so there exists some $M>0$ such that in $\Om$,
        \begin{equation*}
            -\De_p u_\la -\De_q u_\la =\la(u_\la^\si - \frac{1}{u_\la^\ba}) \leq\ \la u_\la^\si \leq\ \la \Phi_\la^\si \leq\ \la \Phi_1^\si\ \leq \la M
        \end{equation*}
        Let $v_\la$ be the solution of 
        \begin{align}
            -\De_p v_\lambda -\De_q v_\lambda &= \lambda M \ in \ \Om \nonumber\\
            v_\lambda &= 0 \ on \ \pa \Om \nonumber
        \end{align}
    Then by weak comparison principle, $u_\lambda \leq v_\lambda$ in $\Om$. Now by Proposition 10 of \cite{papageorgiou2020nonlinear}, as $\lambda \rightarrow{0+}$, $v_\lambda \rightarrow 0$ in $C^1_0(\overline{\Om})$. Thus by comparison principle any such solution $u_\la$ must be negative which is a contradiction. This completes the proof of Theorem \ref{max_thm}.
\hfill\qed

Now to get a deeper understanding of the solution of equation \eqref{sup_est} we define ${\tilde \Phi}_\la:=\la^{\frac{-1}{p-1-\si}}\Phi_\la$ so that
\begin{align}\label{eqn for limit}
        -\De_p \tilde{\Phi}_\la - \ga \De_q \tilde{\Phi}_\la &= \tilde{\Phi}_\la^\si \ \text{in}\ \Om\nonumber \\
        \tilde{\Phi}_\la &=0 \ \text{on} \ \pa\Om
\end{align}
where $\ga=\la^{\frac{q-p}{p-1-\si}}$ which tends to $0$ as $\la \rightarrow \infty$. Note that there exists a constant $M_0>0$, independent of $\la$, such that $\|\tilde{\Phi}_\la\|_{L^p(\Om)}\leq M_0$. The uniform  $L^\infty$ upper bound we obtain here is true for any $1<q<p<\infty.$\par
\begin{lemma}\label{l_infinity}
    There exists a positive constant $M$, independent of $\la$, such that $\|\tilde{\Phi}_\la\|_\infty \leq M$.
\end{lemma}

\begin{proof}
We adapt the idea of De Giorgi - Stampachia iteration method to obtain the uniform $L^\infty$ estimate. First of all we define $v^{(\la)}:=\eo\tilde{\Phi}_\la$ where $\eo=\frac{\delta^{\frac{1}{p}}}{M_0},$ so that $\|v^{(\la)}\|_{L^p(\Om)}\leq \delta^{\frac{1}{p}}$ where  $\delta$ is to be chosen later.  Hereafter, for convenience, we simply write $v$ for $v^{(\la)}$. Then $v$ satisfies the following BVP:
    \begin{align}\label{v_gamma}
        -\frac{1}{\eo^{p-1-\si}} \De_p v -\frac{\ga}{\eo^{q-1-\si}} \De_q v &= v^\si\ \mbox{in}\ \Om\nonumber\\
        v &=0\ \mbox{on}\ \pa\Om
    \end{align}
    For $k\in\mathbb{N}$, we define $C_k:= 1- 2^{-k}$, $v_k:= v-C_k$, $w_k:= v_k^+$ and $U_k:= \|w_k\|_{L^p(\Om)}^p$. Then, $0\leq w_k\leq |v|+1$ in $\Om$ and $w_k=0$ on $\pa\Om.$ Clearly, $w_k\in W^{1,p}_0(\Om)$ and $w_{k+1}\leq w_k, \forall k\in \mathbb{N}$. Also we note that $\lim_{k\rar\infty} w_k= (v-1)^+$ and so by Lebesgue Dominated Convergence Theorem, $\lim_{k\rar\infty} U_k= \int_\Om [(v-1)^+]^p\ dx$.\par
    Now define $A_k:= \frac{C_{k+1}}{C_{k+1}-C_k}= 2^{k+1}-1,\ k\in \mathbb{N}$. Then $v< A_k w_k$ on the set $\{w_{k+1}>0\}$. Taking $w_{k+1}$ as test function in the weak formulation of \eqref{v_gamma}, we get,
    \begin{equation*}
        \frac{1}{\eo^{p-1-\si}} \int_\Om |\na w_{k+1}|^p\ dx + \frac{\ga}{\eo^{q-1-\si}} \int_\Om |\na w_{k+1}|^q\ dx = \int_\Om v^\si w_{k+1}\ dx
    \end{equation*}
    Since $\ga>0$ we have,
    \begin{align}\label{grad_ineq}
        \|\na w_{k+1}\|_{L^p(\Om)}^p &\leq \eo^{p-1-\si} \int_\Om v^\si w_{k+1}\ dx \leq \int_{\{w_{k+1}>0\}} v^\si w_{k+1}\ dx\nonumber\\
        &\leq \int_{\{w_{k+1}>0\}} (1+v^{p-1})w_{k+1}\ dx\nonumber\\
        &\leq \int_{\{w_{k+1}>0\}} (1+A_k^{p-1}w_k^{p-1})w_k\ dx\nonumber\\
        &\leq  |\{w_{k+1}>0\}|^{1-\frac{1}{p}} U_k^{\frac{1}{p}} + 2^{(k+1)(p-1)}U_k
    \end{align}
    Now, as in (3.64) of Proposition 3.4 from \cite{bisci2016variational}, we get,
    \begin{equation}\label{eq46}
        U_k \geq 2^{-(k+1)p}|\{w_{k+1}>0\}|
    \end{equation}
   Using the above estimate in \eqref{grad_ineq} we get 
\begin{equation}\label{eq47}
    \|\nabla w_{k+1}\|_{L^p(\Omega)}^p \leq c \cdot 2^{(k+1)(p-1)}U_k.
\end{equation}
    By H\"older's inequality with the exponents $\frac{N}{N-p}$ and $\frac{N}{p}$ we have, 
    \begin{equation}\nonumber
        U_{k+1}\leq \|w_{k+1}\|^{p}_{L^{\frac{Np}{N-p}}} |\{w_{k+1}>0\}|^{\frac{p}{N}} \leq c_0 \,(2^{\frac{p^2}{N}+p})^k \,U_k^{1+\frac{p}{N}}
    \end{equation}
    To derive the last inequality in the previous line we have used \eqref{eq46} and \eqref{eq47}. Therefore for some $C>1$ and $\al=\frac{p}{N},$
    \begin{equation}
        U_{k+1}\leq \, C^{k} U_k^{1+\al}
    \end{equation}
    Now we shall define $\delta:=C^{\frac{-1}{\alpha^2}}.$ Then following the induction argument as in \cite{bisci2016variational}, Proposition 3.4 we have $\lim_{k \rightarrow 0} U_k =0.$ We also know that $lim_{k\rightarrow \infty} U_k = \int_\Om [(v-1)^+]^p\ dx,$ thus $(v-1)^+=0$ a.e. in $\Om$ or $v\leq 1$ a.e. in $\Om$. Hence $\tilde{\Phi}_\la \leq \frac{1}{\eo}$ a.e. in $\Om$ and since $\Tilde{\Phi}_\lambda>0$ we have the required $L^\infty$ estimate. 
\end{proof}

\begin{theorem}\label{est}
Let $\Phi_\la$ be the solution of the equation as given in \eqref{sup_est}, $2<q<p<\infty$ and $\la_0$ be a positive real number. Then there exist positive constants $c_1,c_2, \la_0$ such that for all $\la\geq\la_0$ and for all $x\in\Om$,
    \begin{equation}
    c_1 \la^{\frac{1}{p-1-\si}}d(x)\leq \Phi_\la (x)\leq c_2 \la^{\frac{1}{p-1-\si}}d(x).
    \end{equation}
\end{theorem}
\begin{proof}
    From lemma \ref{l_infinity}, there exists a positive constant $M_1$ such that $\tilde{\Phi}_\la^\si \leq M_1$ a.e. in $\Om$. Let $w$ be the weak solution of
    \begin{align}
        -\De_p w - \ga \De_q w &= M_1 \ \text{in} \ \Om\nonumber \\
        w &=0 \ \text{on} \ \pa\Om\nonumber
    \end{align}
    Therefore, by weak comparison principle,
    $\tilde{\Phi}_\la(x) \leq w(x)$, $\forall x\in \Om$.
    Now following the ideas of the proof in Proposition 2.1 of \cite{Indulekha}, we get that for some positive constant $c_2$, independent of $\gamma,$  and for all $x\in \Omega,$ $\tilde{\Phi}_\la(x) \leq w(x) \leq c_2 d(x).$ i.e 
    \begin{align}\label{upper_bound}
      \Phi_\la (x) &\leq c_2 \la^{\frac{1}{p-1-\si}}d(x)
    \end{align}
    In order to obtain a lower bound we resonate some of the ideas used in Lemma \ref{lemma_sub}.
    Define 
    \begin{equation}
        \xi_\la:=\epsilon \la^{\frac{1}{p-1-\sigma}} (\phi_1 + \phi_1^\al) = \epsilon \la^{r_0} (\phi_1 + \phi_1^\al)
    \end{equation}
    where $1<\al<2$, $r_0= \frac{1}{p-1-\si}$ and $\phi_1$ is the first eigen function of $-\De$ such that $\|\phi_1\|_\infty=1$. \\
    \underline{Claim: }  $\xi_\la$ is a subsolution of \eqref{sup_est} for some admissible range of $\la$.
    \par Proceeding same as lemma \ref{lemma_sub} we get, for $s= p, q$,
    \begin{equation}\label{s_laplacian}
        -\De_s \xi_\la = \epsilon^{s-1}\la^{r_0(s-1)} \phi_1^{\al-2} L_s(\phi_1)
    \end{equation}
    where $L_s(\phi_1)= (1+\al\phi_1^{\al-1})^{s-1}\phi_1^{2-\al}[-\De_s \phi_1] - \al(\al-1)(s-1)|\na\phi_1|^s(1+\al\phi_1^{\al-1})^{s-2}$ and since $s>2$, there exist $\delta>0$ and $m>0$ such that $L_p(\phi_1)<-m$ and $L_q(\phi_1)<0$ in $\Om_\delta$. Therefore, $\xi_\la$ being non-negative by definition, we have,
    \begin{equation}\label{near_bdry}
        -\De_p \xi_\la -\De_q \xi_\la \leq -\frac{\epsilon^{s-1}\la^{r_0(s-1)}m}{\phi_1^{2-\al}} \leq \la \xi_\la^\si\ \mbox{in } \Om_\delta
    \end{equation}
    On the other hand, in $\Om\setminus\Om_\delta$, there exists $\mu>0$ such that $\phi_1>\mu$. Then since $\De_s \phi_1 \in L^\infty(\Om)$ as $2<q<p$, from \eqref{s_laplacian} we get,
    \begin{equation}\label{c_s}
        -\De_s \xi_\la \leq \frac{\epsilon^{s-1}\la^{r_0(s-1)}c_s}{\overline{\mu}} \mbox{ in } \Om\setminus\Om_\delta
    \end{equation}
    for some positive constants $c_s$ (which depends on $s$ only) and $\overline{\mu}= \mu^{2-\al}$. Since in $\Om\setminus\Om_\delta$, $\phi_1>\mu$, so $\xi_\la\geq\epsilon\la^{r_0} \mu$ and consequently $\epsilon^\sigma\la^{r_0\si}\mu^\si\leq\xi_\la^\si.$
   First we choose $\epsilon\in(0,1)$ such that   $\frac{2c_s\epsilon^{s-1-\si}}{\overline{\mu}\mu^\si}\leq 1$ for $s=p,q$. Also note that $1+r_0(\si-s+1)\geq 0$ for $s=p,q$. Therefore from \eqref{c_s} we have for $\la\geq 1$,
    \begin{equation*}
        -\De_s \xi_\la \leq \frac{\epsilon^{s-1} \la^{r_0(s-1) }c_s}{\overline{\mu}} \leq \frac{\la}{2} (\epsilon^\si \la^{r_0\si}\mu^\si) \leq \frac{\la\xi_\la^\si}{2}
    \end{equation*}
    \begin{equation}\label{int}
        \Rightarrow -\De_p \xi_\la -\De_q \xi_\la \leq \la \xi_\la^\si \mbox{ in } \Om\setminus\Om_\delta
    \end{equation}
    Thus \eqref{near_bdry} and \eqref{int} together establish that for $\la\geq 1$, $\xi_\la$ is a positive subsolution of \eqref{sup_est}. Again $\Phi_\la$ is defined to be the positive solution of the same equation \eqref{sup_est}. Then by the comparison lemma \ref{comparison} we have for $\la\geq1$, $\xi_\la\leq \Phi_\la$ or 
    $$\Phi_\la \geq \epsilon\la^{r_0}(\phi_1+\phi_1^\al)\, \geq \,\epsilon\la^{\frac{1}{p-1-\si}} \phi_1 $$
     Now since $\phi_1$ is the first eigenfunction of $-\De$, so taking $\la_0=1$, there exists a positive constant $c_1$ such that for all $\la\geq \la_0$,
    \begin{equation}\label{lower_bound}
        c_1 \la^{\frac{1}{p-1-\si}} d(x) \leq \Phi_\la (x) \mbox{ for all } x\in\Om.
    \end{equation}
    Thus combining \eqref{upper_bound} and \eqref{lower_bound} we get the desired estimate.

\end{proof}

\section{Appendix}

\begin{lemma}\label{comparison}
    Suppose $f:[0,\infty)\rightarrow \mathbb{R}$ is a continuous function with $f(0)=0$ and $f(s)s^{1-q}$ is non-increasing in $(0,\infty)$ and $u_1, u_2 \in W^{1,p}_0(\Om) \cap C(\overline{\Om})$ satisfy
    \begin{align}\label{u1}
        -\De_p u_1-\De_q u_1 &\leq f(u_1)\ \text{in}\ \Om\nonumber\\
        u_1 &>0\ \text{in}\ \Om
    \end{align}
    and
    \begin{align}\label{u2}
        -\De_p u_2-\De_q u_2 &\geq f(u_2)\ 
         \text{in}\ \Om\nonumber\\
        u_2 &>0\ \text{in}\ \Om
    \end{align}
   Then $u_1\leq u_2$ \text{in} $\overline{\Om}$.
\end{lemma}

\begin{proof}
    Define $E=\{ x\in \Om : u_1(x) > u_2(x)\}$ and consider $w_1, w_2 \in W^{1,q}_0(\Om)$ defined by
    \begin{center}
        $w_1= \frac{(u_1^q -u_2^q)^+}{u_1^{q-1}}$ and $w_2= \frac{(u_1^q -u_2^q)^+}{u_2^{q-1}}$
    \end{center}
    Then from \eqref{u1} and \eqref{u2}, we get,
    \begin{eqnarray}
        \int_E |\na u_1|^{p-2} \na u_1 \cdot \na w_1\ dx +  \int_E |\na u_1|^{q-2} \na u_1 \cdot \na w_1\ dx \leq \int_E f(u_1)w_1\ dx \nonumber\\
        \int_E |\na u_2|^{p-2} \na u_2 \cdot \na w_2\ dx +  \int_E |\na u_2|^{q-2} \na u_2 \cdot \na w_2\ dx \geq \int_E f(u_2)w_2\ dx \nonumber
    \end{eqnarray}
    Subtracting the latter equation from the former, we have,
    \begin{multline}\label{inequality}
        \int_E \Bigl[ |\na u_1|^{p-2} \na u_1 \cdot \na w_1 - |\na u_2|^{p-2} \na u_2 \cdot \na w_2 \Bigr]\ dx + \int_E \Bigl[ |\na u_1|^{q-2} \na u_1 \cdot \na w_1 \\
        - |\na u_2|^{q-2} \na u_2 \cdot \na w_2 \Bigr]\ dx \leq \int_E \Bigl[\frac{f(u_1)}{u_1^{q-1}} - \frac{f(u_2)}{u_2^{q-1}} \Bigr]\bigl(u_1^q -u_2^q\bigr)\ dx
    \end{multline}
    Then from lemma 4.2 in \cite{lindqvist1990equation}, for $q\geq 2$, we get,
    \begin{multline}\label{I2}
        \int_E \Bigl[ |\na u_1|^{q-2} \na u_1 \cdot \na w_1 - |\na u_2|^{q-2} \na u_2 \cdot \na w_2 \Bigr]\ dx\\
        \geq \frac{1}{2^{q-1}-1}\int_E \frac{1}{(u_1^q + u_2^q)}\big| u_1 \na u_2 - u_2 \na u_1 \big|^q\ dx \geq 0
    \end{multline}
    and from Theorem 1.2 of \cite{bobkov2020generalized}, 
    \begin{equation*}
        |\na u_1|^{p-2} \na u_1\cdot \na\Bigl(\frac{u_2^q}{u_1^{q-1}}\Bigr) \leq \frac{q}{p}|\na u_2|^p + \frac{p-q}{p}|\na u_1|^p
    \end{equation*}
    so that
    \begin{equation}\label{I1}
        \int_E \Bigl[ |\na u_1|^{p-2} \na u_1 \cdot \na w_1 - |\na u_2|^{p-2} \na u_2 \cdot \na w_2 \Bigr]\ dx \geq 0
    \end{equation}
    Now as $f(s)s^{1-q}$ is non-increasing, so in \eqref{inequality} the $R.H.S.\leq 0$. Therefore, combining \eqref{inequality}, \eqref{I2} and \eqref{I1}, we get,
    \begin{equation*}
        0\leq \frac{1}{2^{q-1}-1} \int_E \frac{1}{(u_1^q + u_2^q)}\big| u_1 \na u_2 - u_2 \na u_1 \big|^q\ dx \leq 0
    \end{equation*}
    \begin{equation*}
        \Rightarrow | u_1 \na u_2 - u_2 \na u_1 |=0\ in\ E
    \end{equation*}
    so that if $E$ is non-empty, then on each connected component $E_i$ of $E$, $u_1=c_i u_2$ but since $u_1 =u_2$ on $\pa E_i \cap \Om$, so $u_1=u_2$ in $E_i$ which contradicts the definition of $E$. Hence, $E=\phi$ or in other words, $u_1(x)\leq u_2(x)$ in $\overline{\Om}$.
\end{proof}

\bibliographystyle{plain}
\bibliography{References}	

\end{document}